\documentclass [12pt]{amsart}
\usepackage{amsmath,amssymb}
 \usepackage{amsthm, amsfonts}
 \usepackage{enumerate}

\usepackage{cite}
\newtheorem{theorem}{Theorem}[section]

\theoremstyle{definition}
\newtheorem{definition}[theorem]{Definition}
\newtheorem{example}[theorem]{Example}

\theoremstyle{remark}

\numberwithin{equation}{section}

\begin{document}

\title [{{ On $gr$-quasi-semiprime submodules}}]{ On $gr$-quasi-semiprime submodules}

 \author[{{K. Al-Zoubi and S. Alghueiri }}]{\textit{Khaldoun Al-Zoubi* and Shatha Alghueiri}}

\address
{\textit{Khaldoun Al-Zoubi, Department of Mathematics and
Statistics, Jordan University of Science and Technology, P.O.Box
3030, Irbid 22110, Jordan.}}
\bigskip
{\email{\textit{kfzoubi@just.edu.jo}}}
\address
{\textit{ Shatha Alghueiri, Department of Mathematics and
Statistics, Jordan University of Science and Technology, P.O.Box
3030, Irbid 22110, Jordan}}
\bigskip
{\email{\textit{ghweiri64@gmail.com}}}

\maketitle
\date{}
\begin{abstract}
 Let $G$ be a group. A ring $R$ is called a graded ring (or $G$-graded ring)
if there exist additive subgroups $R_{\alpha }$ of $R$ indexed by the
elements $\alpha \in G$ such that $R=\bigoplus_{\alpha \in G}R_{\alpha }$
and $R_{\alpha }R_{\beta }\subseteq R_{\alpha \beta }$ for all $\alpha $, $%
\beta \in G$. If an element of $R$ belongs to $h(R)=\cup _{\alpha \in
G}R_{\alpha }$, then it is called a homogeneous. A Left $R$-module $M$ is
said to be \textit{a graded }$R$\textit{-module} if there exists a family of
additive subgroups $\{M_{\alpha }\}_{\alpha \in G}$ of $M$ such that $%
M=\bigoplus_{\alpha \in G}M_{\alpha }$ and $R_{\alpha }M_{\beta }\subseteq
M_{\alpha \beta }$ for all $\alpha ,\beta \in G.$ Also if an element of $M$
belongs to $\cup _{\alpha \in G}M_{\alpha }=h(M)$, then it is called a
homogeneous. A submodule $N$ of $M$ is said to be \textit{a graded submodule
of }$M$ if $N=\bigoplus_{\alpha \in G}(N\cap M_{\alpha }):=\bigoplus_{\alpha
\in G}N_{\alpha }$. Let $G$ be a group with identity $e$. Let $R$ be a $G$%
-graded commutative ring and $M$ a graded $R$-module. A proper graded
submodule $S$ of $M$ is said to be \textit{a graded semiprime (}shortly $gr$%
\textit{-semiprime) submodule} if whenever $r^{n}m\in S$ where $r\in h(R)$, $%
m\in h(M)$ and $n\in Z^{+}$, then $rm\in S.$ In this work, we introduce the
concept of graded quasi-semiprime (shortly $gr$-quasi-semiprime) submodule
as a generalization of $gr$-semiprime submodule and give some basic
properties of these classes of graded submodules. We say that a proper
graded submodule $S$ of $M$ is a $gr$-quasi-semiprime submodule if $%
(S:_{R}M)=\{r\in R:rM\subseteq S\}$ is a $gr$-semiprime ideal of $R$.
\end{abstract}



 \section{Introduction and Preliminaries}
Throughout this paper all rings are commutative with identity and all
modules are unitary.

 Graded semiprime submodules of graded modules over graded commutative rings, have been introduced and studied in \cite{1, 5, 7, 12}. Also, the
concept of graded semiprime ideal was introduced by Lee and Varmazyar \cite{7}  and studied
in \cite{4}.

Recently, K. Al-Zoubi, R. Abu-Dawwas and I. Al-Ayyoub in \cite{1}
introduced and studied the concept of graded semi-radical of graded submodules in graded modules.

Here, we introduce the concept of graded quasi-semiprime submodules of
graded modules over a commutative graded rings as a generalization of graded
semiprime submodules and investigate some properties of these classes of
graded submodules.
Let $R$ be a $G$-graded ring, $M$ a graded $R$-module and $N$ a graded
submodule of $M$. Then $(N:_{R}M)$ is defined as $(N:_{R}M)=\{r\in
R:rM\subseteq N\}.$ It is shown in \cite[Lemma 2.1]{3} that if $N$ is a graded
submodule of $M$, then $\ (N:_{R}M)\ $ is a graded ideal of $R$. The
annihilator of $M$ is defined as $(0:_{R}M)$ and is denoted by $Ann_{R}(M)$.

A proper graded submodule $N$ of $M$ is said to be \textit{a graded
semiprime submodule} if whenever $r\in h(R)$, $m\in h(M)$ and $n\in Z^{+}$
with $\ r^{n}m\in N$, then $rm\in N$, (see \cite{5}.) A proper graded
ideal $I$ of $R$ is said to be graded semiprime ideal if whenever $r,s\in
h(R)$ and $n\in Z^{+}$ with $r^{n}s\in I$, then $rs\in I$, (see \cite{4}.)
For more information about the properties of graded rings and graded modules see \cite{6, 8, 9, 10}.



 \section{Results}
\begin{definition}
 Let $R$ be a $G$-graded ring and $M$ a graded $R$-module. A proper graded submodule $N$ of $M$ is said to be \emph{a graded
quasi-semiprime submodule of $M$} if $(N:_{R}M)$ is a graded semiprime
ideal of $R$.
\end{definition}

\begin{theorem}
Let $R$ be a $G$-graded ring, $M$ a graded $R$-module and $N$ a proper
graded submodule of $M$. If $N$ is a graded semiprime submodule of $M,$ then
$N$ is a graded quasi-semiprime submodule of $M$.
\end{theorem}
\begin{proof}
 By \cite[Theorem 2.4]{1}.
 \end{proof}
The next example shows that a graded quasi-semiprime submodule is not necessarily graded semiprime submodule.
\begin{example}
 Let $G=%
\mathbb{Z}
_{2}$, $R=%
\mathbb{Z}
$ be a $G$-graded ring with $R_{0}=%
\mathbb{Z}
$ and $R_{1}=\{0\}$. Let $M=%
\mathbb{Z}
\times
\mathbb{Z}
$ be a graded $R$-module with $M_{0}=%
\mathbb{Z}
\times \{0\}$ and $M_{1}=\{0\}\times
\mathbb{Z}
$. Now, consider a submodule $N=4%
\mathbb{Z}
\times \{0\}$ of $M$. Then it is a graded submodule and $(N:_{R}M)=\{0\}$ is
a graded semiprime ideal of $R,$ and so $N$ is a graded quasi-semiprime
submodule of $R.$ But the graded submodule $N$ is not graded semiprime
submodule of $M,$ since $2^{2}(3,0)\in N$ but $2(3,0)\notin N.$
\end{example}
\begin{example}
 Let $G=%
\mathbb{Z}
_{2}$, $R=%
\mathbb{Z}
$ be a $G$-graded ring with $R_{0}=%
\mathbb{Z}
$ and $R_{1}=\{0\}$. Let $M=%
\mathbb{Z}
_{8}$ be a $G$-graded $R$-module with $M_{0}=%
\mathbb{Z}
_{8}$ and $M_{1}=\{0\}$. Now, consider a submodule $N=<4>$ of $M.$ Then it is
a graded submodule and $(N:_{R}M)=4%
\mathbb{Z}
$  is not a graded semiprime ideal of $R$ since $2^{2}1=4\in 4%
\mathbb{Z}
$ but $2\cdot 1=2\notin 4%
\mathbb{Z}
.$ Then $N$ is not graded quasi-semiprime submodule of $M.$
 \end{example}
Recall that a graded $R$-module $M$ is called \emph{a graded multiplication} if for
each graded submodule $N$ of $M$, we have $N=IM$ for some graded ideal $I$ of $R$. If $N$
is graded submodule of a graded multiplication module $M$, then $N=(N:_{R}M)M$.

\begin{theorem}
Let $R$ be a $G$-graded ring, $M$ a graded multiplication $R$%
-module and $N$ a proper graded submodule of $M$. Then $N$ is a graded
quasi-semiprime submodule of $M$ if and only if $N$ is a graded semiprime
submodule of $M$.
\end{theorem}

\begin{proof}
By \cite[Theorem 2.5]{1}.
 \end{proof}
\begin{theorem}
Let $R$ be a $G$-graded ring, $M$ a graded multiplication $R$-module and $N$
a proper graded submodule of $M$. Then the following statements are
equivalent:
\begin{enumerate}[\upshape (i)]

\item $N$ is a graded quasi-semiprime submodule of $M.$

\item If whenever $I^{k}M\subseteq N,$ where $I$ is a graded ideal of $R$ and
$k\in
\mathbb{Z}
^{+},$ then $IM\subseteq N$.
\end{enumerate}
\end{theorem}
\begin{proof}
$(i)\Rightarrow (ii)$  By Theorem 2.5 and \cite[Proposition 2.6]{5}

$(ii)\Rightarrow (i)$ Let $r^{k}s\in (N:_{R}M)$ where $r,s\in h(R)$ and $%
k\in
\mathbb{Z}
^{+}.$ So $r^{k}sM\subseteq N.$ Let $I=(rs)$ be a graded ideal of $R$
generated by $rs$. Then $I^{k}M\subseteq N.$ By our assumption we
have $IM=(rs)M\subseteq N.$ This yields that $rs\in (N:_{R}M)$. So $(N:_{R}M)$
is a graded semiprime ideal of $R.$ Therefore $N$ is a graded
quasi-semiprime submodule of $M.$
\end{proof}
Recall that a proper graded ideal $I$ of a $G$-graded ring $R$ is said to be
\textit{a graded prime ideal }if whenever $r,s\in h(R)$ with $rs\in I$, then
either $r\in I$ or $\ s\in I$ (see \cite{11}.) A proper graded ideal $J
$ of $R$ is said to be \textit{a graded primary ideal }if whenever $r,s\in
h(R)$ with $rs\in J$, then either $r\in J$ or $\ s^{n}\in J$ for some $n\in
\mathbb{Z}
^{+}$ (see \cite{11}.)

\begin{theorem}
Let $R$ be a $G$-graded ring, $M$ a graded $R$-module and $N$ a
graded quasi-semiprime submodule of $M$. If $(N:_{R}M)$ is a graded primary
ideal of $R$, then $(N:_{R}M)$ is a graded prime ideal of $R$.
\end{theorem}
\begin{proof}
Suppose that $(N:_{R}M)$ is a graded primary ideal of $R$.
Let $rs\in (N:_{R}M)$ and $r\notin (N:_{R}M).$ Then $s\in Gr((N:_{R}M))$ as $%
(N:_{R}M)$ is a graded primary ideal of $R.$ Hence $s^{k}\in (N:_{R}M)$ for
some $k\in
\mathbb{Z}
^{+}.$ Since $(N:_{R}M)$ is a graded semiprime ideal of $R,$ we have $s\in
(N:_{R}M)$. Therefore $(N:_{R}M)$ is a graded prime ideal of $R.$
\end{proof}

Let $R$ be a $G$-graded ring, $M$ a graded $R$-module and $N$ a graded
submodule of $M$. The graded envelope submodule $RGE_{M}(N)$ of $N$ in $M$
is a graded submodule of $M$ generated by the set $GE_{M}(N)=\{rm:r\in
h(R),m\in h(M)$ such that $r^{n}m\in N$ for some $n\in
\mathbb{Z}
^{+}\}$ (see \cite[Definition 1]{2}.)
\begin{theorem}
Let $R$ be a $G$-graded ring, $M$ a graded
multiplication $R$-module and $N$ a proper graded submodule of $M$. Then $N$
is a graded quasi-semiprime submodule of $M$ if and only if $N=RGE_{M}(N).$
\end{theorem}
\begin{proof}
Suppose that $N$ is a graded quasi-semiprime submodule of $M$. Then $N$ is
a graded semiprime submodule of $M$ by Theorem 2.5.
Clearly, $N\subseteq RGE_{M}(N)$. Now, let $x\in GE_{M}(N).$ Then $x=rm$ for
some $r\in h(R),$ $m\in h(M)$ and there exists $k\in
\mathbb{Z}
^{+}$ such that $r^{k}m\in N.$ Then $rm\in N$ as $N$ is a graded semiprime
submodule of $M$. Hence $GE_{M}(N)\subseteq N.$ This yields that $%
RGE_{M}(N)\subseteq N.$ Thus $N=RGE_{M}(N).$ Conversely, suppose that $%
N=RGE_{M}(N)$. Let $r\in h(R)$, $m\in h(M)$ and $k\in
\mathbb{Z}
^{+}$ such that $r^{k}m\in N$, so by the definition of the set $GE_{M}(N)$
we have $rm\in GE_{M}(N).$ Then $rm\in N$ as $GE_{M}(N)\subseteq RGE_{M}(N)=N$, so $N$ is a graded semiprime submodule of $M$. Therefore $N$ is a graded
quasi-semiprime submodule of $M$ by Theorem 2.2.
\end{proof}

Let $R$ be a $G$-graded ring and $M$, $M^{\prime }$ be two graded $R$-modules. Let $%
f:M\rightarrow M^{\prime }$ be an $R$-module homomorphism. Then $f$ is said
to be \emph{a graded homomorphism }if $f(M_{\alpha })\subseteq M_{\alpha }^{\prime }
$ for all $\alpha $ $\in G$ (see \cite{10}.)
\begin{theorem}
Let $R$ be a $G$-graded ring, $M$, $M^{\prime }$ be two
graded $R$-modules and $f:M\rightarrow M^{\prime }$a graded epimorphism.
 \begin{enumerate}[\upshape (i)]

\item If $N$ is a graded quasi-semiprime submodule of $M$ such that $%
ker(f)\subseteq N$, then $f(N)$ is a graded quasi-semiprime submodule of $%
M^{\prime }$.

\item If $N^{\prime }$ is a graded quasi-semiprime submodule of $M^{\prime }$%
, then $f^{-1}(N^{\prime })$ is a graded quasi-semiprime submodule of $M.$

\end{enumerate}
\end{theorem}

\begin{proof}
$(i)$ Suppose that $N$ is a graded quasi-semiprime submodule of $M$ and $%
ker(f)\subseteq N.$ It is easy to see that $f(N)\neq
M^{\prime }.$ Now let $r^{k}s\in (f(N):_{R}M^{\prime })$ where $r,s\in
h(R)$ and $k\in
\mathbb{Z}
^{+},$ it follows that, $r^{k}sM^{\prime }\subseteq f(N).$ Then $%
r^{k}sM^{\prime }=r^{k}sf(M)=f(r^{k}sM)\subseteq f(N)$ since $f$ is an
epimorphism. This yields that $r^{k}sM\subseteq N$ since $ker(f)\subseteq N,$
i.e., $\ r^{k}s\in (N:_{R}M).$ Since $N$ is a graded quasi-semiprime
submodule of $M$, we get $rs\in (N:_{R}M),$ i.e., $rsM\subseteq N.$ Hence $%
f(rsM)=rsf(M)=rsM^{\prime }\subseteq f(N),$ i.e., $rs\in (f(N):_{R}M^{\prime
}).$ Therefore, $f(N)$ is a graded quasi-semiprime submodule of $M^{\prime
}. $

$(ii)$ Suppose that $N^{\prime }$ is a graded quasi-semiprime submodule of $%
M^{\prime }$. It is easy to see that $f^{-1}(N^{\prime
})\neq M.$ Let $r^{k}s\in (f^{-1}(N^{\prime }):_{R}M)$ where $r,s\in h(R)$\
and $k\in
\mathbb{Z}
^{+},$ it follows that, $r^{k}sM\subseteq f^{-1}(N^{\prime }).$ Then $%
r^{k}sf(M)=r^{k}sM^{\prime }\subseteq \ N^{\prime },\ $i.e.$,$ $r^{k}s\in
(N^{\prime }:_{R}M^{\prime }).$ Then $rs\in (N^{\prime }:_{R}M^{\prime })$
as $N^{\prime }$ is a graded quasi-semiprime submodule of $M^{\prime }$. So $%
rsM^{\prime }=rsf(M)=f(rsM)\subseteq N^{\prime }.$ It follows that $%
rsM\subseteq f^{-1}(N^{\prime }).$ So $rs\in (f^{-1}(N^{\prime }):_{R}M).$
Therefore $f^{-1}(N^{\prime })$ is a graded quasi-semiprime submodule of $M.$
\end{proof}
\begin{theorem}
 Let $R$ be a $G$-graded ring, $M$ a graded $R$-module
and $K$ a proper graded submodule of $M$. If $N$ is a graded quasi-semiprime
submodule of $M$ with $N\subseteq K$ and $(N:_{R}M)$ is a graded maximal
ideal of $R$, then $K$ is a graded quasi-semiprime submodule of $M$.
\end{theorem}
\begin{proof}
Suppose that $N\subseteq K,$ it follows that $(N:_{R}M)\subseteq (K:_{R}M).$
By \cite[Lemma 2.1]{3}, $(K:_{R}M)$ is a
proper graded ideal of $R$. Then $(N:_{R}M)=(K:_{R}M)$ as $(N:_{R}M)$ is a
graded maximal ideal of $R$.\ This yields that $(K:_{R}M)$ is a graded
semiprime ideal of $R$. Therefore $K$ is a graded quasi-semiprime submodule
of $M$.
\end{proof}
\begin{theorem}
Let $R$ be a $G$-graded ring, $M$ a graded $R$-module
and $N$ and $K$ be two graded quasi-semiprime submodules of $M.$ Then $N\cap
K$ is a graded quasi-semiprime submodule of $M$.
\end{theorem}
\begin{proof}
 Let $r^{k}s\in (N\cap K:_{R}M)$\ where $r,s\in h(R)$\ and $%
k\in
\mathbb{Z}
^{+}.$ This yields that $r^{k}s\in (N:_{R}M)\cap (K:_{R}M).$ Since $%
(N:_{R}M)\ $and $(K:_{R}M)$ are graded semiprime ideals of $R$, we have $%
rs\in (N:_{R}M)\cap (K:_{R}M)$ and so $rs\in (N\cap K:_{R}M).$ Therefore $%
N\cap K$ is a graded quasi-semiprime submodule of $M$.

\end{proof}

Let $R$ be a $G$-graded ring and $M$ be a graded $R$-module, $M$ is called a
graded semiprime module if $(0)$ is a graded semiprime submodule of $M$.

\begin{definition}
 Let $R$ be a $G$-graded ring and $M$ be a graded $R$-module. Then $M$ is
said to be \emph{a graded quasi-semiprime module }if $Ann_{R}N$ is a graded
semiprime ideal of $R$, for every non-zero graded submodule $N$ of $M$.
\end{definition}

 \begin{theorem}
 Let $R$ be a $G$-graded ring and $M$ be a graded $R$%
-module. If $M$ is a graded semiprime module, then $M$ is a graded
quasi-semiprime module.
 \end{theorem}
 \begin{proof}
 Suppose that $M$ is a graded semiprime module. Then
$0$ is a graded semiprime submodule of $M.$ Now, Let $N$ be a non-zero
graded submodule of $M$ and $r^{k}s\in $ $Ann_{R}N$  where $r,s\in h(R)$\
and $k\in
\mathbb{Z}
^{+}.$ It follows that $r^{k}sN=0.$ Then $rsN=0$ as $0\ $ is a graded
semiprime submodule of $M.$ Hence $rs\in Ann_{R}N,$ it follows that $%
Ann_{R}N$ is a graded semiprime ideal of $R.$ Therefore $M$ is a graded
quasi-semiprime module.
 \end{proof}


\bigskip\bigskip\bigskip\bigskip

\end{document}